\documentclass[a4paper,12pt]{article}
\usepackage{xcolor}
\usepackage[a4paper, total={7in, 8in}]{geometry}

\usepackage{cancel}
\usepackage{pgfplots}
\usepackage{amsmath}
\usepackage{mathrsfs}

\usepackage{amsmath, amsthm, amsfonts,amssymb,stmaryrd}

\usepackage{pst-node} 
\usepackage[extra]{tipa}
\usepackage{pstricks}
\usepackage{tikz}
\usetikzlibrary{arrows}
\usepackage[spanish,USenglish]{babel}

\usepackage [T1]{fontenc}
\usepackage{amssymb,amsthm,amsmath,amsfonts}
\usepackage{graphicx}
\usepackage[affil-it]{authblk}
\usepackage{url}
\usepackage{subfigure}
\usepackage{comment}

\usepackage{colortbl} 
\usepackage{tikz}
\usepackage[colorlinks=true, allcolors=blue]{hyperref} 
\usetikzlibrary{shapes.geometric}
\usepackage{graphicx}

\hypersetup{colorlinks=true, linkcolor=blue}

\usepackage{enumerate}
\usetikzlibrary{shapes.geometric}

\newtheorem{thm}{Theorem}[section]
\newtheorem{mthm}{Main Theorem}
\newtheorem{cor}[thm]{Corollary}
\newtheorem{lem}[thm]{Lemma}
\newtheorem{prop}[thm]{Proposition}
\theoremstyle{definition}
\newtheorem{notn}[thm]{Notation}

\theoremstyle{remark}
\newtheorem{rem}[thm]{Remark}

\title{The change of vertex energy when joining trees.}
\author[1]{Octavio Arizmendi} 
\author[2]{Saylé Sigarreta}
\date{\today}

\affil[1]{ Department of Probability and Statistics, Centro de Investigación en
Matemáticas, Guanajuato, México}
\affil[2]{Department of Probability and Statistics,Benemérita Universidad Autónoma de Puebla, Puebla, México}

\begin{document}

\maketitle

\abstract{In this manuscript we study how the vertex energy of a tree is affected when joined with a bipartite graph.  We find an alternating pattern with respect to the coalescence vertex: the energy decreases for vertices located at odd distances and increases for those located at even distances. }




\textbf{MSC classes}: 05C05, 05C09

\section{Introduction}

Gutman introduced in 1978, the energy of a graph, \cite{b2}, motivated by notion of  $\pi$-electron energy in
conjugated hydrocarbon molecules.   Forty years later, inspired by Non-Commutative Probability, Arizmendi and Ju\'arez-Romero \cite{b8}  introduced the concept of energy of a vertex (or vertex energy). Since we can calculate the energy of a graph by adding the individual energies of its vertices, it follows that the energy of a vertex must be understood as the contribution of this vertex to the energy of the graph. In this sense, studying vertex energies constitutes a tool to achieve a better understanding of the energy of a graph.  The basics of the theory were developed in \cite{b9}, where the authors derived fundamental inequalities and gave examples and counterexamples of natural conjectures for the vertex energy. 

Taking into account the developed theory and in search of an interpretation that would obey intuition and allow us to understand the interaction between local and global energy in graphs, in \cite{b11} the authors defined the graph energy game. An important conclusion is that, by considering the  vector of vertex energies as the payoff vector of the cooperative game, an interpretation of how the total energy is distributed among the vertices of the graph can be offered.

Motivated by the above, in the present work we continue the research in this direction. Specifically, the main objective of this manuscript is to study the effect produced in the vertex energies of a tree after being merged with a bipartite graph. Our main contribution is to show that the change in vertex energy has an alternating behavior, when considering the distance with the merging vertex. Let us state our main result, the definitions involved in the statement can be seen in Section \ref{s1}. 

\begin{mthm}
Let $T$ be a tree and $B$ be a bipartite graph.  Consider the coalescence of $T\circ B$ where  we have identified the vertices $v$ in $T$ and $u$ in $B$ where $deg_{B}(u)\geq 1$.
Let $w$ be a vertex in $T$ and $\hat{w}$ its corresponding copy in $T\circ B$ then

\begin{itemize}
    \item If $d(v,w)$ is odd, $\mathcal{E}_T(w) >\mathcal{E}_{T\circ B}(\hat{w}).$
    \item If $d(v,w)$ if is even, $\mathcal{E}_T(w) < \mathcal{E}_{T\circ B}(\hat{w})$.
\end{itemize}

\end{mthm}

As we explain below, in Section \ref{s2}, this is consistent with the intuition that vertex energy acts as a payoff in a cooperative game.


\section{Preliminaries}\label{s1}
In this section we introduce basic results on Graph Theory, Graph Energy and Vertex Energy for further reference.

\subsection{Basic definitions in Graph Theory}

In this work we will only consider simple finite undirected graphs. A  simple undirected graph is a pair $G=(V,E)$, where $E\subseteq V\times V$ is a set such that $(v,v)\notin E$ for all $v\in V$, and $(v,w)\in E$ implies that $(w,v)\in E$. $V$ is called the vertex set and $E$ the edge set. 
The degree $deg_{G}(v)$ of a vertex $v$ is the number of edges at $v$. The distance $d_{G}(u, v)$ in $G$ of two vertices $u, v$ is the length of a
shortest $u-v$ path in $G$. A connected graph that does not contain any cycle is called a tree and a graph that does not contain cycles of odd length is called bipartite.
Given $G$ and $H$ two graphs with disjoint vertex sets, $u \in G$ and $v \in H$, the graph $G \circ H$ is known as the coalescence of $G$ and $H$ with respect to $u$ and $v$. It is constructed from copies of $G$ and $H$ by identifying the vertices $u$ and $v$. 

Let $G=(V,E)$ be a graph of order $n$, where $V=\{v_{1},v_{2}, \dots, v_{n}\}$. The adjacency matrix $A (G)=(a_{ij})_{n \times n}$ of $G$ is defined by $a_{ij} = 1$ if $(v_{i},v_{j})\in E$ and 0 otherwise. The characteristic polynomial of $G$, denoted by $\phi_{G} (x)$,  is defined as the characteristic polynomial of $A(G)$, that is,
$\phi_{G}(x) =det (xI-A)=:\sum_{k=0}^{n}a_{k}x^{n-k}$,
 where $I$ is the identity matrix of size $n \times n$. In particular, it is verified that, $\phi_{G \circ H} (x)=\phi_{G} (x)\phi_{H \setminus v}(x)+\phi_{G \setminus u} (x)\phi_{H}(x)-x\phi_{G \setminus u}(x)\phi_{H \setminus v}(x).$ Also, if $G$ is a bipartite graph, then its characteristic polynomial is of the form $\phi_{G}(x)=\sum_{k=0}^{\lfloor n / 2\rfloor}(-1)^k b_{2k}(G)x^{n-2k}$, where $b_{2k}(G)=|a_{2k}(G)|$.

\subsection{Graph energy} 

The energy of $G$, denoted by $\mathcal{E}(G)$, is defined as
$\mathcal{E}(G)=Tr(|A(G)|),$
where $|A(G)| = (A (G)A(G)^{\ast})^{1/2}$. In particular, when $G$ is bipartite,  the  Coulson's integral formula \cite{b5} motivates the following definition of a quasi-order for bipartite graphs. For two bipartite graphs $G_1$ and $G_2$, we define the quasi-order $\preceq$ and write $G_1 \preceq G_2$ if $b_{2 k}\left(G_1\right) \leq b_{2 k}\left(G_2\right)$ for all $k$. If, in addition, at least one of the inequalities $b_{2 k}(G_1) \leq b_{2 k}(G_2)$ is strict, then we write $G_1 \prec G_2$.  The importance of this order lies in  the fact that it allows to compare the energy of two graphs. 

While developing the main results of this paper, we will need the following results.

\begin{lem} \label{l4.4} (Theorem 1.3, \cite{b14})
 Let $u v$ be an edge of $G$. Then
$$
\phi_{G}(x)=\phi_{G \setminus uv}(x)-\phi_{G \setminus \{u, v\}}(x)-2 \displaystyle\sum_{C \in \mathscr{C}(u v)}\phi_{G \setminus C}(x),$$
where $\mathscr{C}(u v)$ is the set of cycles containing $u v$.
\end{lem}

\begin{lem}\label{p4}(Theorem 4.18, \cite{b14})
Let $G$ and $H$ be graphs. Then
 $\mathcal{E}(G \circ H)\leq \mathcal{E}(G)+\mathcal{E}(H).$
\end{lem}

\begin{lem} (Theorem 4.20, \cite{b14}) \label{p3}
If $F$ is an edge cut of a simple graph G, 
$\mathcal{E}(G \setminus F )\leq \mathcal{E}(G).$  
\end{lem}

\subsection{Energy of a vertex}

The energy of the vertex $v_{i} \in V(G)$, which is denoted by $\mathcal E_{G}(v_{i})$, was defined in \cite{b8}, as follows: $\mathcal E_{G}(v_{i}) = |A(G)|_{ii}$, for $i = 1, 2, \dots,n$. Equivalently, 

\begin{equation}\label{e1}
   \mathcal{E}_G\left(v_i\right)=\sum_{j=1}^n p_{i j}\left|\lambda_j\right|, \quad i=1, \ldots, n, 
\end{equation}
where $\lambda_j$ denotes the $j$-eigenvalue of the adjacency matrix of $G$ and the weights $p_{i j}$ satisfy
$$
\sum_{i=1}^n p_{i j}=1 \text { and } \sum_{j=1}^n p_{i j}=1 .
$$
More precisely, $p_{ij}=u_{ij}^2$ where  $U = (u_{ij} )^n_{i,j=1}$ is the  orthogonal matrix whose columns are given by the eigenvectors of $A(G)$. On the other hands,
\begin{equation}\label{moments}
M_k(G, i)=\sum_{j=1}^n p_{i j} \lambda_j^k, \quad i=1, \ldots, n,
\end{equation}
where $M_k(G, i)$ is equal to the number of $v_i-v_i$ walks in $G$ of length $k$.

By definition and guided by the results obtained for $\mathcal{E}(G)$, in \cite{b10},  a Coulson-type integral formula to calculate the energy of a vertex was found and as a consequence, it allowed to elucidate the structural interpretation of the energy of a vertex and its relationship with the energy of a bipartite graph, since, it provided a method to compare the energy of two vertices in such graphs. 
\begin{lem}\label{p9} (Theorem 4, \cite{b10})
Let $G$ be a graph, then for a vertex $v_i$ in $G$,
$$
\mathcal{E}_G\left(v_i\right)=\frac{1}{\pi} \int_{\mathbb{R}} 1-\frac{\mathbf{i} x \phi\left(G-v_i ; \mathbf{i} x\right)}{\phi(G ; \mathbf{i} x)} d x.
$$
\end{lem}

\begin{lem}\label{p7} (Lemma 6, \cite{b10})
 Let $G$ a bipartite graph and $v,w \in G$. If $G \setminus w \succeq G \setminus v$, then $\mathcal{E}_{G}(w) \leq \mathcal{E}_{G}(v)$.  Moreover if $G  \setminus w \neq G \setminus v$ then $\mathcal{E}_{G}(w) < \mathcal{E}_{G}(v)$.
\end{lem}
The following result which is a consequence of Lemma \ref{p7} will be often used in our proofs below. 
\begin{lem}\label{p2} (Theorem 7, \cite{b10})
 Let $G_1$ and $G_2$ two disjoint bipartite graphs with $w \in G_2$ and $v \in G_1$. If $G_1 \cup\left(G_2 \setminus w\right) \succeq G_2 \cup\left(G_1 \setminus v\right)$, then $\mathcal{E}_{G_2}(w) \leq \mathcal{E}_{G_1}(v)$. Moreover if $G_2\setminus w \neq G_1 \setminus v$ then $\mathcal{E}_G(w)<\mathcal{E}_G(v)$.

\end{lem}

The remaining two lemmas will also be used in what follows.

\begin{lem} \label {p5} (Proposition 3.9, \cite{b9})
Let $G$ be a bipartite graph with parts $V_{1}$ and $V_{2}$. Then $$\displaystyle\sum_{v \in V_{1}} \mathcal{E}_G(w)=\displaystyle\sum_{v \in V_{2}} \mathcal{E}_G(w).$$
\end{lem}

\begin{lem} \label {p10} (Theorem 3, \cite{b6}) Let $v_i$ and $v_j$ be connected vertices of a simple (undirected) graph G. Then
$\mathcal{E}\left(v_i\right) \mathcal{E}\left(v_j\right) \geq 1$.
\end{lem}

\section{Main Results}\label{s2}

In this section we prove the main theorem. To prepare for the main part of the proof, we need some notation.

\begin{notn}
Let $T$ be a tree and consider a given path $P=\{v_1\sim v_2 \sim \cdots \sim v_n\}$ in $T$. If $n \geq 2$, then for $1\leq i\leq n-1$, let $e_i=v_iv_{i+1}$.

\begin{enumerate}

\item For $n=1$, $\tilde T$ is the empty graph. For $n \geq 2$, $\tilde T$ denotes the connected component of $T\setminus v_1$ which contains $v_2$. 
\item $A_i$ denotes the connected component of $T\setminus e_i$ which contains $v_i$.
\item For 
$i=1$, $\tilde A_i$ is the empty graph. For $i\geq 2$ (which implies that $n \geq 3$), $\tilde A_i$ denotes the connected component of $A_i\setminus v_1$ that contains $v_{2}$. 
\end{enumerate}
\end{notn}

The following simple lemma is crucial for the proof of our main theorem. 

\begin{lem}\label{L1}
Let $T$ be a tree and $P=\{v_1 \sim v_2 \sim \cdots \sim v_n\}$ be a path in $T$ with $n\geq 2$. Then for $1\leq i \leq n-1$. 
The following order is satisfied in the quasi-order $\preceq$:
\begin{itemize}
    \item If $i$ is even, $T\cup\tilde A_i \prec \tilde T\cup A_i.$
\item If $i$ is odd,  $T\cup\tilde A_i \succ \tilde T\cup A_i$.
\end{itemize}

\end{lem}
\begin{proof}
We will get the proof via mathematical induction on $i$. Our base case will be, $i=1,2$. First, for $i=1$, by Lemma \ref{l4.4}  applied to  $T$ with $e_{1}$, we can observe that
$$b_{2k}(T\cup \tilde A_1)=b_{2k}(T)=b_{2k}(\tilde T \cup A_{1})+ b_{2(k-1)}(\widetilde {\textit{\textbf{T}}}\cup (A_2\setminus\{v_{1},v_{2}\})),$$
where $\textit{\textbf{T}}=\tilde T$ and $\textit{\textbf{P}}=\{u_{1}=v_{2}\sim u_{2}=v_{3}\sim \dots \sim u_{n-1}=v_{n}\}$; which completes this part. For $i=2,$ by Lemma \ref{l4.4}  applied to $A_2$ and $T$ with $e_{1}$, respectively, we obtain 
\begin{center}
 $b_{2k}(\tilde T\cup A_2)=b_{2k}(\tilde T \cup \tilde A_{2} \cup A_{1})+b_{2(k-1)}(\textit{\textbf{T}} \cup (A_2\setminus\{v_{1},v_{2}\}))$,
\end{center}
and
\begin{center}
 $b_{2k}(T\cup\tilde A_2)=b_{2k}(\tilde T \cup \tilde A_{2} \cup A_{1})
+b_{2(k-1)}(\widetilde {\textit{\textbf{T}}}\cup {\textit{\textbf{A}}}_{1}\cup (A_2\setminus\{v_{1},v_{2}\}))$.
\end{center}
Then, if now with $\textit{\textbf{T}}$ we apply the same argument used when $i=1$, we obtain the conclusion. In the following, we will show the statement for $i$. By using Lemma \ref{l4.4} for $A_{i}$ and $T$ with $e_{1}$, the following equalities hold

\begin{center}
    $b_{2k}(\tilde T \cup A_i)=b_{2k}(\tilde T \cup \tilde A_{i}\cup A_{1})+b_{2(k-1)}(\textit{\textbf{T}}\cup \widetilde {\textit{\textbf{A}}}_{i-1} \cup A_{2}\setminus \{v_{1},v_{2}\}$),
\end{center}
and
\begin{center}
    $b_{2k}( T\cup \tilde A_i)=b_{2k}(\tilde T \cup \tilde A_{i}\cup A_{1})+b_{2(k-1)}(\widetilde {\textit{\textbf{T}}} \cup {\textit{\textbf{A}}}_{i-1} \cup A_{2}\setminus \{v_{1},v_{2}\})$.
\end{center}
 Finally, applying the induction hypothesis for $\textit{\textbf{T}}\cup \widetilde {\textit{\textbf{A}}}_{i-1}$ and $\widetilde {\textit{\textbf{T}}} \cup {\textit{\textbf{A}}}_{i-1} $ the proof is complete.

\end{proof}
\begin{rem}\label{oc}
 The result seen above is also valid for bipartite graphs when the path $P$ is formed solely by edge cuts.
\end{rem}
Now we can prove the main theorem of the paper that we state again for the convenience of the reader.

\begin{thm}\label{t1}
Let $T$ be a tree and $B$ be a bipartite graph.  Consider the coalescence of $T\circ B$ where  we have identified the vertices $v$ in $T$ and $u$ in $B$ where $deg_{B}(u)\geq 1$.
Let $w$ be a vertex in $T$ and $\hat{w}$ its corresponding copy in $T\circ B$ then

\begin{itemize}
    \item If $d_{T}(v,w)$ is odd, $\mathcal{E}_T(w) >\mathcal{E}_{T\circ B}(\hat{w}).$
    \item If $d_{T}(v,w)$ if is even, $\mathcal{E}_T(w) < \mathcal{E}_{T\circ B}(\hat{w})$.
\end{itemize}

\end{thm}
\begin{proof}

By Lemma \ref{p2} it is enough to prove that 
\begin{center}
    $(T\circ B )\cup(T \setminus w)  \prec T \cup((T\circ B)\setminus \hat{w})$,
\end{center}
when $d_{T}(v,w)$ is odd and,

\begin{center}
    $T \cup((T\circ B)\setminus \hat{w}) \prec  (T\circ B) \cup(T \setminus w)$,
\end{center}
when $d_{T}(v,w)$ is even. Firstly, for $d_{T}(v,w) \geq 1$, let $P=\{v_{1}=v\sim v_{2}\sim \dots \sim v_{d(v,w)+1}=w\}$ be the path in $T$ connecting the vertices $v$ and $w$ and note that 

\begin{center}
  $T \cup((T\circ B)\setminus \hat{w})=T \cup (A_{d(v,w)}\circ B) \cup Q,$  
\end{center}
and 
\begin{center}
  $(T\circ B) \cup(T\setminus w)=(T \circ B) \cup A_{d(v,w)} \cup Q,$  
\end{center}
where $Q$ is the disjoint union of the connected components of $T \setminus v_{d(v,w)+1}$ that do not contain $ v_{d(v,w)}$ and $A_{d(v,w)}\circ B$ is  the coalescence with respect
to $v$ and $u$. Now, by computing the characteristic polynomial of $A_{d(v,w)} \circ B$ and $T \circ B$, respectively, we have that
\\
\\
$b_{2k}(T \cup (A_{d(v,w)}\circ B))=$
$$b_{2k}(T \cup A_{d(v,w)}\cup B \setminus u)+b_{2k}(T \cup A_{d(v,w)} \setminus v \cup B))-b_{2k}(T \cup A_{d(v,w)} \setminus v \cup B \setminus u )),$$
and
\\
\\
$b_{2k}((T \circ B) \cup A_{d(v,w)}))=$\\
$$b_{2k}(T \cup A_{d(v,w)} \cup B \setminus u)+b_{2k}(T \setminus v  \cup A_{d(v,w)} \cup B))-b_{2k}(T \setminus v \cup A_{d(v,w)}  \cup B \setminus u )).$$
\\
Since,  the first term is the same in both expressions and $B \setminus u  \prec B$, all boil down to compare 
$b_{2k}(T \cup A_{d(v,w)} \setminus v )$ and $b_{2k}(T \setminus v  \cup A_{d(v,w)})$. Rewriting the last two expressions in the following way
$$b_{2k}(T \cup A_{d(v,w)} \setminus v )=b_{2k}(T \cup \tilde A_{d(v,w)} \cup A_{1} \setminus v ),$$
and
$$b_{2k}(T \setminus v  \cup A_{d(v,w)})=b_{2k}(\tilde T \cup A_{d(v,w)} \cup A_{1} \setminus v ).$$
By Lemma \ref{L1}, this part is completed. Finally, if $w=v$ then 
\begin{center}
  $T \cup((T\circ B)\setminus \hat{v})=T \cup B\setminus \hat{v}\cup T \setminus v,$  
\end{center}
and 
\begin{center}
  $(T\circ B) \cup(T\setminus v)=(T \circ B) \cup T \setminus v.$  
\end{center}

Hence,  the statement is verified,  if in a similar way we calculate the characteristic polynomial of $T \circ B$.
\end{proof}



 The central problem of cooperative games is to examine how the benefits obtained with the coalition will be distributed among all players. Now, since given a graph, we have that $\mathcal E_{G}(G)=\sum_{i=1}^{n}\mathcal E_{G}(v_{i})$. Then, it is natural to think this as a partitioning problem. With this approach, in \cite{b11}, the authors define the energy game as follows: Let $G=(V, E)$ be a simple undirected graph, for each subset $S \subseteq V$, they define the
characteristic function as $w(S):=$ $\mathcal{E}(I(S))$ where $I(S)$ is the graph induced by $S$. Furthermore, they show that the vertex energy is a payoff vector that belongs to the core.
Thus, one potential theoretical explanation for this behavior may be derived through the application of a Game Theory perspective: Consider the path $P=\{v_1 \sim v_2\cdots \sim v_n\}$ in $T$, where we merge a bipartite graph $B$ in $v_1$. First, the player who is used for the merger, increases its options to establish alliances. It is expected then to increase its valuation as well, since, it will have a larger negotiating margin. On the other hand $v_{2}$ now has more competition, that is, the new neighbours of $v_{1}$ hence lower negotiating margin leading to a decrease in its payoff. By applying inductively the same reasoning on the remaining players in the path, we should obtain the alternating behavior observed in Theorem \ref{t1} (Figure \ref{f5}).
 \begin{figure}[ht]
\begin{center}
\usetikzlibrary{arrows}
\pagestyle{empty}

\definecolor{ffdazd}{rgb}{0,0,0}
\begin{tikzpicture}[line cap=round,line join=round,>=triangle 45,x=1cm,y=1cm]
\clip(-4.186359329241893,-0.8610910899166204) rectangle (6.607662709172672,2.7313465854071053);
\draw [line width=0.8pt] (0,0)-- (1,0);
\draw [line width=0.8pt] (1,0)-- (2,0);
\draw [line width=0.8pt] (2,0)-- (3,0);
\draw [line width=0.8pt] (4,0)-- (5,0);
\draw [rotate around={0:(3,0)},line width=0.8pt] (3,0) ellipse (3.073142765785426cm and 0.6664881536075677cm);
\draw [rotate around={0:(-2,0)},line width=0.8pt] (-2,0) ellipse (2.098207090202712cm and 0.6344075924647653cm);
\draw (-0.2,-0.1723340735442095) node[anchor=north west] {$\hat{v}$};
\draw (4.912869614726483,0.09) node[anchor=north west] {$\hat{w}$};
\draw (-0.5,1) node[anchor=north west] {$B \circ T$};
\draw (3.1,0.14479145917402286) node[anchor=north west] {$\dots$};

\draw (5.2,0.5) node[anchor=north west] {$\uparrow$};
\draw (4.040791614671647,0.6) node[anchor=north west] {$\uparrow$};
\draw (4.2,0.6) node[anchor=north west] {$\downarrow$};
\draw (5.044504029829099,0.5) node[anchor=north west] {$\downarrow$};
\begin{scriptsize}
\draw [fill=ffdazd] (0,0) circle (4pt);
\draw (0.04,0.38) node[anchor=north west] {$\uparrow$};
\draw [fill=ffdazd] (1,0) circle (4pt);
\draw (1.02,0.38) node[anchor=north west] {$\downarrow$};
\draw [fill=ffdazd] (2,0) circle (4pt);
\draw (2.02,0.38) node[anchor=north west] {$\uparrow$};
\draw [fill=ffdazd] (3,0) circle (4pt);
\draw (3.02,0.38) node[anchor=north west] {$\downarrow$};
\draw [fill=ffdazd] (4,0) circle (4pt);
\draw [fill=ffdazd] (5,0) circle (4pt);

\end{scriptsize}
\end{tikzpicture}
\caption{Graphical representation of the behavior observed in the Theorem \ref{t1}.} 
    \label{f5}
 \end{center}

\end{figure}
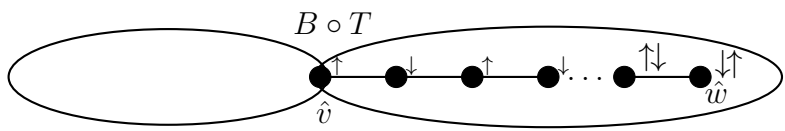 

\begin{rem}\label{oc1}
    According to Remark \ref{oc} and the proof seen above, the result shown in Theorem \ref{t1}  is also valid  for $B_{1} \circ B_{2}$ respect to $u_{i} \in V(B_{i})$ with $B_{i}$ bipartite graphs if we take either $w=u_{i}$ or any $w \in V(B_{i}) \setminus u_{i}$ connected to $u_{i}$ by means of a path formed by edge cuts (which tells us that the path is unique). Moreover, the energy of vertices in $B_{i}$ that are not connected with $u_{i}$ is the same after coalescence.
\end{rem}

Next, by directly applying the Theorem \ref{t1} and Remark \ref{oc1} we can obtain the following corollary; which gives us information about the behavior of vertex energy when we perform the opposite operation, that is, when we remove edges or vertices, and as expected it shows a contrary behavior. 
\begin{cor}\label{c0}
Let $T$ be a tree and $B$ be a bipartite graph.  Consider the coalescence  $T\circ B$ where  we have identified the vertices $v$ in $T$ and $u$ in $B$, respectively. Given $v_{1}v_{2} \in E(T)$, we will denote by $T_{v_i}$
the connected component of $T \setminus v_1v_2$ containing $v_i$. Then taken a vertex  $w$ in $T_{v_{i}}$ we have that:

\begin{itemize}
    \item If $d(w,v_{i})$ is odd, $\mathcal{E}_{(T \circ B)\setminus v_{2}v_{1} }(\hat{w}) >\mathcal{E}_{T \circ B}(w).$
    \item If $d(w,v_{i})$ is even, $\mathcal{E}_{(T \circ B)\setminus v_{2}v_{1}}(\hat{w}) < \mathcal{E}_{T \circ B}(w)$.
\end{itemize}
\end{cor}

It is important to highlight that in the previous corollary we covered all the vertices in $T$, since, $T=T_{v_{1}} \cup T_{v_{2}}$. In particular, when $B$ is the graph formed by an isolated vertex, it shows us what is the behavior in the trees after removing an edge; and as consequence, after eliminating  vertices.

The preceding theorems compare the energies before and after coalescence. Another pertinent issue regarding the merging vertex is to compare the energy of the vertex after coalescence with their individual energies before coalescence.  This is exactly what the following theorem addresses.

\begin{thm}\label{t3}
Given $G$ and $H$ two bipartite graphs with disjoint vertex sets. Let $u \in G$, $v \in H$ and $w$ its copy in $G \circ H$, then

$$\mathcal{E}_{G\circ H}(w) \leq \mathcal{E}_{G}(u)  + \mathcal{E}_{ H}(v) .$$

Moreover, equality is attained if and only if either $u$ is an isolated vertex of $G$ or $v$ is an isolated
vertex of $H$ or both.

\end{thm}
\begin{proof}
By Lemma \ref{p9}, we have

   \begin{align*}
\mathcal{E}_{G \circ H}(v) &= \frac{1}{\pi} \int_{\mathbb{R}} \left(1 - \frac{i x \phi(G \circ H - v, i x)}{\phi(G \circ H, i x)}\right) d x \\
&= \frac{1}{\pi} \int_{\mathbb{R}} \left(1 - \frac{i x \phi(G - v,ix) \phi(H - v;ix)}{\phi(G \circ H ; i x)}\right) d x \\
&= \frac{1}{\pi} \int_{\mathbb{R}} \left(1 + \frac{\phi(G \circ H, i x)}{\phi(G \circ H, i x)} - \frac{\phi(G;ix) \phi(H - v;ix)}{\phi(G \circ H;ix)} - \frac{\phi(H;ix) \phi(G - v)}{\phi(G \circ H;ix)}\right) d x \\
&= \frac{1}{\pi} \int_{\mathbb{R}} \left(1 - \frac{\phi(G;ix) \phi(H - v;ix)}{\phi(G \circ H;ix)} \right) d x + \frac{1}{\pi} \int_{\mathbb{R}} \left(1 - \frac{\phi(H;ix) \phi(G - v;ix)}{\phi(G \circ H ;ix)}\right) d x.
\end{align*}

For the first integral notice that,
\begin{align*}
\frac{\phi(G;ix) \phi(H-v;ix)}{\phi(G \circ H;ix)} &= \frac{(i x)^{n+m-1} \sum_{k \geq 0}(-1)^k b_{2 k}(G)(i x)^{-2 k} \sum_{k \geq 0}(-1)^k b_{2 k}(H-v)(i x)^{-2 k}}{(i x)^{n+m-1} \sum_{k\geq 0}(-1)^k b_{2 k}(G \circ H)(i x)^{-2 k}} \\
&= \frac{\sum_{k \geq 0} b_{2 k}(G) x^{-2 k} \sum_{k \geq 0} b_{2 k}(H-v) x^{-2 k}}{\sum_{k \geq 0} b_{2 k}(G \circ H) x^{-2 k}},
\end{align*}
where $|V(G)|=n$ and $|V(H)|=m$. Besides,
$$
\frac{\phi(H-v;ix) i x}{\phi(H;ix)}=\frac{\sum_{k\geq 0} b_{2 k}(H-v) x^{-2 k}}{\sum_{k\geq 0} b_{2 k}(H) x^{-2 x} }, 
$$
and
$$
\left(\sum_{k \geq 0} b_{2 k}(H) x^{-2 k}\right)\left(\sum_{k\geq 0} b_{2 k}(G) x^{-2 k}\right)=\sum_{k \geq 0} b_{2 k}(G \cup H) x^{-2 k}. $$
Since $b_{2 k}(G \cup H) \geq b_{2 k}(G \circ H) \geq 0$, it is verified that
$$
\frac{\sum_{k>0} b_{2 k}(G) x^{-2 k} \sum_{k>0} b_{2 k}(H-v) x^{-2 k}}{\sum_{k>0} b_{2 k}(G \circ H) x^{-2 k}} \geq \frac{\sum_{k>0} b_{2 k}(H-v) x^{-2 k}}{\sum_{k>0} b_{2 k}(H) x^{-2 k}}.
$$
Thus, \begin{equation}\label{ eq:join2}\frac{1}{\pi} \int_{\mathbb{R}} \left(1 - \frac{\phi(G) \phi(H - v)}{\phi(G \circ H)}\right) d x  \leq \frac{1}{\pi} \int_{\mathbb{R}} \left(1 - \frac{\phi(H - v)}{\phi( H)}\right) d x =\mathcal{E}_H(v)
\end{equation}
By conducting a similar process with the other summand we obtain
\begin{equation}\label{ eq:join3}\frac{1}{\pi} \int_{\mathbb{R}} \left(1 - \frac{\phi(H) \phi(G - v)}{\phi(G \circ H)}\right) d x \leq \frac{1}{\pi} \int_{\mathbb{R}} \left(1 - \frac{\phi(H - v)}{\phi( H)}\right) d x =\mathcal{E}_G(v)\end{equation}
which completes the proof.
\end{proof}

\begin{rem}
    Note that the above result is a local version of Lemma \ref{p4} and when is contrasted with Theorem \ref{t1} and Remark \ref{oc1} tells us how much the energy of the vertex (where  coalescence is made) can increase. 
\end{rem}
\subsection{ Successive coalescence}
Now, we consider the effect of connecting successively a sequence of bipartite graphs. Later we will study the limiting behavior.

\begin{cor}\label{c1}
Let $T$ be a tree, $\{B_{n}\}_{n\geq 0}$ a sequence of bipartite graphs, $v \in V(T)$ and $\{b_{n}\}_{n\geq 0}$ a sequence of vertices such that $b_{n} \in V(B_{n})$. Now, we will construct a sequence of graphs $\{G_{n}\}_{n \geq 0}$ using the following recursive method.
\begin{itemize}
    \item For $n=1$, by identifying the vertex $v$ in $T$ and the vertex $b_{1}$ in $B_{1}$, we define $G_{1}=T\circ B_{1}$ and let $v_{1}$ be the copy of $v$ in $G_{1}$.
      \item For $n \geq 2$, let  $v_{n-1}$ be the copy of $v$ in $G_{n-1}$, we define $G_{n}=G_{n-1} \circ B_{n}$ where we have identified the vertices $v_{n-1}$ and $b_{n}$, respectively.
\end{itemize}  
Let $w$ be a vertex in $T$ and $\hat{w}$ its corresponding copy in $G_{n}$ then the following results hold:

\begin{itemize}
    \item If $d(v,w)$ is even, then $\{\mathcal{E}_{G_{n}}(\hat{w})\}_{n\geq 0}$ is an increasing sequence and $\mathcal{E}_{ T\setminus v}(w)$ is a upper bound except when $d(v,w)=0$.
    \item If $d(v,w)$  is odd, then $\{\mathcal{E}_{G_{n}}(\hat{w})\}_{n\geq 0}$ is an decreasing sequence and $\mathcal{E}_{T\setminus v}(w)$ is a lower bound.
\end{itemize}
\end{cor}

\begin{proof}
Since by definition $G_{n}=G_{n-1}\circ B_{n}$, if we apply Theorem \ref{t1} it is easy to show that $\{\mathcal{E}_{G_{n}}(\hat{w})\}_{n}$ is an increasing and a decreasing sequence when $d(v,w)$ is even and odd, respectively. On the other hand, given $w \in V(T\setminus v)$, denoting by  $P=\{u_{1}=v\sim u_{2}\sim \dots \sim u_{d(v,w)+1}=w\}$ the unique path in $T$ which connects $v$ and $w$. If we remove from $G_{n}$ the copy of $u_{1}u_{2}$, by Corollary \ref{c0} the result follows.

\end{proof}

Corollary \ref{c1} has an important implication which is not  obvious: Since for each $w \in V(T \setminus v)$, $\{\mathcal{E}_{G_{n}}(\hat{w})\}_{n}$ is a monotone sequence which is bounded, it is a convergent sequence. A natural question is of course what is this limit. For the particular case where $B_n=K_2$, and then $G_n=S_{n+1}\circ T$, we can calculate such limit. This is the content of next theorem.

\begin{prop}\label{t2}

Let $v_{c}$ and  $v_{h}$ be the center and the leaves of the start $S_{n+1}$ in the graph $S_{n+1} \circ T$ with $n \in \mathbb{N}$. Then, let $v \in V(S_{n+1} \circ T)\setminus\{v_{c},v_{h}\}$, $\mathcal{E}_{S_{n+1} \circ T}(v) \rightarrow \mathcal{E}_{ T\setminus v_{c}}(v)$ when $n \rightarrow \infty.$ Moreover,  $\mathcal{E}_{S_{n+1} \circ T}(v_{h}) \rightarrow 0$ and $\mathcal{E}_{S_{n+1} \circ T}(v_{c}) \rightarrow \infty$   when $n \rightarrow \infty.$ 
\end{prop}

\begin{proof}
 First, note that applying Theorem \ref{t1} we have that $\mathcal{E}_{S_{n+1} \circ T}(v_{h}) \leq \frac{1}{\sqrt{n}}=\mathcal{E}_{S_{n+1}}(v_{h})$ and $\mathcal{E}_{S_{n+1}}(v_{c})=\sqrt{n}  \leq \mathcal{E}_{S_{n+1} \circ T}(v_{c})$. Then, $\mathcal{E}_{S_{n+1} \circ T}(v_{h}) \rightarrow 0$ and $\mathcal{E}_{S_{n+1} \circ T}(v_{c}) \rightarrow \infty$ when $n \rightarrow \infty.$
  On the other hand, by Lemmas \ref{p4} and \ref{p3}, we have
\begin{center}
    $2\sqrt{n}=\mathcal{E}(S_{n+1}) \leq \mathcal{E}(S_{n+1} \circ T)- \mathcal{E}(T\setminus v_{c}) \leq \mathcal{E}(S_{n+deg_{T}(v)+1})=2\sqrt{n+deg_{T}(v)}$.
\end{center}

\noindent

If we define $V_{1}=\{v \in V( T \setminus v_{c})~|~d(v,v_{c})~is~ odd\}$ and $V_{2}=\{v \in V( T \setminus v_{c})~|~d(v,v_{c})~is~ even\}$, applying Lemma \ref{p5}, Theorem \ref{t1} and the above inequalities, we obtain

\begin{center}
     $0 \leq \displaystyle\sum_{v \in V_{1}}\mathcal{E}_{S_{n+1} \circ T}(v)- \displaystyle\sum_{v \in V_{1}}\mathcal{E}_{T \setminus v_{c}}(v) \leq \sqrt{n+deg_{v}(T)}-n\mathcal{E}_{S_{n+1} \circ T}(v_{h})$,
\end{center}
\noindent
and
\begin{center}
     $\sqrt{n}-\mathcal{E}_{S_{n+1} \circ T}(v_{c})\leq \displaystyle\sum_{v \in V_{2}}\mathcal{E}_{S_{n+1} \circ T}(v)- \displaystyle\sum_{v \in V_{2}}\mathcal{E}_{T \setminus v_{c}}(v) \leq 0$.
\end{center}

\noindent
Since by Theorem \ref{t1}, 
$$\mathcal{E}_{S_{n+deg_{T}(v)+1}}(v_{h})=\frac{1}{\sqrt{n+deg_{T}(v)}} \leq \mathcal{E}_{S_{n+1} \circ T}(v_{h}),$$ 
and $$\mathcal{E}_{S_{n+1} \circ T}(v_{c})\leq \sqrt{n+deg_{T}(v)}=\mathcal{E}_{S_{n+deg_{T}(v)+1}}(v_{c}) ,$$
then, the following limits hold
\begin{center}
    $\displaystyle\lim_{n \to \infty}(\displaystyle\sum_{v\in V_{1}} (\mathcal{E}_{S_{n+1} \circ T}(v)-\mathcal{E}_{T \setminus v_{c}}(v)))= \displaystyle\sum_{v\in V_{1}} (inf_{n}\{\mathcal{E}_{S_{n+1} \circ T}(v)\}-\mathcal{E}_{T \setminus v_{c}}(v))$=0,
\end{center}

\noindent
and

\begin{center}
     $\displaystyle\lim_{n \to \infty}(\displaystyle\sum_{v\in V_{2}} (\mathcal{E}_{S_{n+1} \circ T}(v)-\mathcal{E}_{T \setminus v_{c}}(v)))= \displaystyle\sum_{v\in V_{2}} (sup_{n}\{\mathcal{E}_{S_{n+1} \circ T}(v)\}-\mathcal{E}_{T \setminus v_{c}}(v))$=0.
\end{center}

\noindent
Because of Corollary \ref{c1}, we know that $\mathcal{E}_{T \setminus v_{c}}(v)$ is a lower bound if $v \in V_{1}$ and  $\mathcal{E}_{T \setminus v_{c}}(v)$ is a upper bound if $v \in V_{2}$, hence we conclude the proof.
\end{proof}

\begin{figure}[ht]
\begin{center}
\usetikzlibrary{arrows}
\pagestyle{empty}

\definecolor{ffdazd}{rgb}{0,0,0}
\begin{tikzpicture}[line cap=round,line join=round,>=triangle 45,x=1cm,y=1cm]
\clip(-1.1398476080868958,-2.257903039183148) rectangle (3.963214829461329,2.085172713451595);
\draw (-0.8297357786612778,0.3) node[anchor=north west] {$\vdots$};
\draw (0.841422413243441,0.2) node[anchor=north west] {$\vdots$};
\draw (-0.9,-1.5) node[anchor=north west] {$S_{n+1}\circ T$}; 
\draw (-0.3817964694909407,1.470112852529334) node[anchor=north west] {\scriptsize $1$};
\draw (-0.8331814656548958,1.2002396482471174) node[anchor=north west] {\scriptsize $2$};
\draw (-1.0881623031826262,0.635388755563408) node[anchor=north west] {\scriptsize $3$};
\draw (-1.1,-0.2) node[anchor=north west] {\scriptsize $n$};

\draw [line width=0.4pt] (0,0)-- (-0.9857364255104072,0.16829646289202096);
\draw [line width=0.4pt,] (0,0)-- (-0.7237375449666366,0.6900753335728396);
\draw [line width=0.4pt] (0,0)-- (-0.29639921487245285,0.9550641368112371);
\draw [line width=0.4pt] (0,0)-- (-0.7213873210309514,-0.6925318281897137);
\draw [line width=0.4pt,dotted] (0,0)-- (0.75,1.04);
\draw [line width=0.4pt] (0.75,1.04)-- (1.544415127528584,1.58);
\draw [line width=0.4pt] (1.544415127528584,1.58)-- (2.019349164467898,1.04);
\draw [line width=0.4pt] (2.019349164467898,1.04)-- (2.599824098504837,1.88);
\draw [line width=0.4pt] (2.599824098504837,1.88)-- (3.123131046613896,1.64);
\draw [line width=0.4pt] (0.75,1.04)-- (1.562005277044855,1.04);
\draw [line width=0.4pt,dotted] (0,0)-- (0.849604221635884,0.36);
\draw [line width=0.4pt] (0.849604221635884,0.36)-- (1.6675461741424802,0.28);
\draw [line width=0.4pt] (1.6675461741424802,0.28)-- (2.441512752858399,0.64);
\draw [line width=0.4pt] (1.6675461741424802,0.28)-- (2.2744063324538257,0.04);
\draw [line width=0.4pt] (2.441512752858399,0.64)-- (2.8636763412489006,0.84);
\draw [line width=0.4pt] (2.8636763412489006,0.84)-- (3.3693931398416885,0.52);
\draw [line width=0.4pt] (2.8636763412489006,0.84)-- (3.5233069481090586,1.08);
\draw [line width=0.4pt,dotted] (0,0)-- (0.81,-0.9);
\draw [line width=0.4pt] (0.81,-0.9)-- (1.65,-0.92);
\draw [line width=0.4pt] (1.65,-0.92)-- (2.29,-0.86);
\draw [line width=0.4pt] (1.65,-0.92)-- (2.433329804294947,-0.3374100040585359);
\draw [line width=0.4pt] (2.433329804294947,-0.3374100040585359)-- (2.95707422732488,-0.27464879376034596);
\draw [line width=0.4pt] (2.433329804294947,-0.3374100040585359)-- (2.8158010605865424,-0.8394996864440553);
\draw [line width=0.4pt] (2.8158010605865424,-0.8394996864440553)-- (3.35,-0.9);
\draw [line width=0.4pt] (3.35,-0.9)-- (3.69,-1.14);
\draw [line width=0.4pt] (3.35,-0.9)-- (3.85,-0.56);

\draw (-0.2,-0.20561146243233705) node[anchor=north west] {$v$};
\begin{scriptsize}
\draw [fill=black] (0,0) circle (2.5pt);
\draw [fill=black] (-0.9857364255104072,0.16829646289202096) circle (2.5pt);
\draw [fill=black] (-0.7237375449666366,0.6900753335728396) circle (2.5pt);
\draw [fill=black] (-0.29639921487245285,0.9550641368112371) circle (2.5pt);
\draw [fill=black] (-0.7213873210309514,-0.6925318281897137) circle (2.5pt);
\draw [fill=black] (0.75,1.04) circle (2.5pt);
\draw [fill=black] (1.544415127528584,1.58) circle (2.5pt);
\draw [fill=black] (2.019349164467898,1.04) circle (2.5pt);
\draw [fill=black] (2.599824098504837,1.88) circle (2.5pt);
\draw [fill=black] (3.123131046613896,1.64) circle (2.5pt);
\draw [fill=black] (1.562005277044855,1.04) circle (2.5pt);
\draw [fill=black] (0.849604221635884,0.36) circle (2.5pt);
\draw [fill=black] (1.6675461741424802,0.28) circle (2.5pt);
\draw [fill=black] (2.441512752858399,0.64) circle (2.5pt);
\draw [fill=black] (2.2744063324538257,0.04) circle (2.5pt);
\draw [fill=black] (2.8636763412489006,0.84) circle (2.5pt);
\draw [fill=black] (3.3693931398416885,0.52) circle (2.5pt);
\draw [fill=black] (3.5233069481090586,1.08) circle (2.5pt);
\draw [fill=black] (0.81,-0.9) circle (2.5pt);
\draw [fill=black] (1.65,-0.92) circle (2.5pt);
\draw [fill=black] (2.29,-0.86) circle (2.5pt);
\draw [fill=black] (2.433329804294947,-0.3374100040585359) circle (2.5pt);
\draw [fill=black] (2.95707422732488,-0.27464879376034596) circle (2.5pt);
\draw [fill=black] (2.8158010605865424,-0.8394996864440553) circle (2.5pt);
\draw [fill=black] (3.35,-0.9) circle (2.5pt);
\draw [fill=black] (3.69,-1.14) circle (2.5pt);
\draw [fill=black] (3.85,-0.56) circle (2.5pt);

\end{scriptsize}
\end{tikzpicture}
\caption{
Graphical representation of the behavior observed in the Theorem \ref{t2}.} 
    \label{f6}
 \end{center}

 \end{figure}
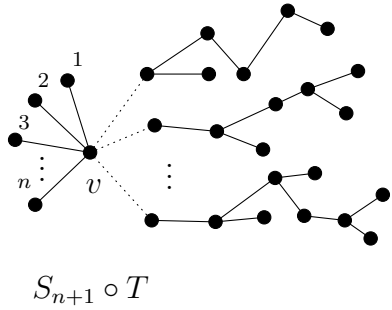

\begin{rem} \label{syi}
    By Lemma \ref{p10} we know that it was sufficient to show that $\mathcal{E}_{S_{n+1} \circ T}(v_{h}) \rightarrow 0$ when $n \rightarrow \infty.$ But, we proved both simultaneously using  Theorem \ref{t1}. Another result in the same direction of Theorem \ref{t1} is Theorem 7.3 in \cite{b9}, however, it is  not applicable  in the case above.
\end{rem}

Here is a possible explanation for the behavior observed in Corollary \ref{c1} and Theorem \ref{t2} by using a game theoretical approach.  By recruiting a new player at each step, the player $v$ increases its value and as a consequence the payoff of the other players is continually affected. Therefore, it seems natural to expect that the fair payoff offered in the long run to the remaining
players is at least the payoff they would receive if $v$ was not in the game,  i.e., what they would receive if they all decide to stop cooperating with $v$ and maintain the other alliances. In search of an adequate payoff, through the above perspective, we assure the players at odd distance that the most their payoff could decrease to would be the amount they would receive if $v$ was not part of the game, thus motivating them to stay in the grand coalition and thus continue to cooperate with $v$. On the other hand, it seems logical to think that players who are at an even distance will remain in the grand coalition as long as we assure them that the minimum their payoff could increase to would be the amount they would receive if $v$ was not part of the game. In conclusion,  whether their payoff at each step decreases or increases,  gave them the same payoff as they would receive  if $v$ was removed from the game and the other alliances were  maintained is strategically  the best option. Since, they end up being given the minimum payoff they "demanded" for remaining in the grand coalition and the payoff they could receive if $v$ was not part of the game and the other alliances were  maintained. In this way, are not tempted to stop cooperating with $v$, which implies to leave the grand coalition. It should be noted that the above tells us  that in the asymptotic behavior the edges incident on $v$ tend to disappear (see Figure \ref{f6}).

\section{Examples}

Given the above, it is valid to wonder about the possibility of finding analogous results to Theorem \ref{t2} for more general graphs $B_{n}$. In this line of analysis, in order to illustrate potential variations in the behavior demonstrated in Theorem \ref{t2}, we examine the graph $G_{n}$ constructed as in Corollary \ref{c1}, taking $T=K_{2}$, $B_{n}=S_{d_{n}+2}$ and $b_{n}$ any leave of $B_{n}$. Here, $\{d_{n}\}_{n \geq 1}$ represents a succession of natural numbers. Inspired by the behavior of the energy of vertex observed in Theorem \ref{t2}, in this case of study, we will analyze whether it is also verified that 
 $\mathcal{E}_{G_{n}}(v) \rightarrow \infty$ y $\mathcal{E}_{G_{n}}(u) \rightarrow 0$ when $n \rightarrow \infty$?

To address this objective, let us examine the following. Given $n \in \mathbb{N}$ if we  define $d(n)=\max\{d_1, d_2, \ldots, d_n\}$, according to the Theorem \ref{t1}, we can state that $\mathcal{E}_{G_{n}}(v) >\mathcal{E}_{H_{n,d(n)}}(v),$ and $\mathcal{E}_{G_{n}}(u)  < 
 \mathcal{E}_{H_{n,d(n)}}(u),$ where $H_{n,d(n)}$ is constructed as $G_n$ but with the stars $S_{d(n)+2}$. Then, using iterately that $H_{n,d(n)}=H_{n-1,d(n)} \circ S_{d(n)+2}$ for $n \geq 1$ we can observe that
$$\phi_{H_{n,d(n)}}(x)=x^{n \cdot (d(n)-1)}(x^{2}-d(n))^{n-1}(x^{4}-(n+d(n)+1)x^{2}+d(n)).$$ Therefore, the spectra of this graph is $\{0,\pm\sqrt{d(n)},\pm \sqrt{\frac{n+d(n)+1 \pm \sqrt{ (n+d(n)+1)^{2}-4d(n)} } {2}  }\} .$

Since there are only 7 distinct eigenvalues, using \eqref{moments} the weights $p_{1 i}$ with $i=1,2,\dots,7$ for $u$ can be calculated  with a 7 by 7 system of linear equations. Obtaining that $p_{11}=p_{12}=p_{13}=0$, 

$$p_{14}=p_{15}=\frac{\sqrt{ (n+d(n)+1)^{2}-4d(n)}-n-d(n)+1  }{4\sqrt{ (n+d(n)+1)^{2}-4d(n)  }},$$
and
$$p_{16}=p_{17}=\frac{\sqrt{ (n+d(n)+1)^{2}-4d(n)}+n+d(n)-1  }{4\sqrt{ (n+d(n)+1)^{2}-4d(n)  }}.$$

Thus, by equation (\ref{e1})

$$\mathcal{E}_{H_{n,d(n)}}(u)=\frac{\sqrt{(n+d(n)+1)^2-4 d(n)}(\sqrt{d(n)}+1)+n+1-d(n)+\sqrt{d(n)}(n+d(n)+1)}{\sqrt{2}\sqrt{(n+d(n)+1)^2-4 d(n)}\sqrt{\sqrt{(n+d(n)+1)^2-4 d(n)}+n+d(n)+1}}.$$

Analyzing the previous expression we can see that,  if $\frac{d(n)}{n}  \rightarrow 0 $ when $n \rightarrow \infty$, then $\mathcal{E}_{H_{n,d(n)}}(u) \rightarrow 0$ when $n \rightarrow \infty$. Regarding the graph $G_{n}$, the above tells us that the asymptotic behavior of $\mathcal{E}_{G_{n}}(v)$ and $\mathcal{E}_{G_{n}}(u)$ depends in part on how it behaves in the limit $d(n)/n$. Revealing, in particular, that if the maximum degree of the graphs coalescing with $K_{2}$ can be "controlled" such that it grows slower than $deg_{G_{n}}(v)=n+1$, then according to Lemma \ref{p10} the behavior holds, i.e., $\mathcal{E}_{G_{n}}(v) \rightarrow \infty$ y $\mathcal{E}_{G_{n}}(u) \rightarrow 0$ when $n \rightarrow \infty$.  As an example, when the sequence $\{d_{n}\}_{n \geq 0}$ is upper bounded. Rewriting the above condition in terms of the monotonicity of the sequence we observe that if it is decreasing, the result is always preserved. On the other hand, if the sequence is increasing, it does so when it is convergent, that is, when it tends to stabilize.

Likewise, the result obtained for the graph $H_{n,d(n)}$ give us examples that allow us to analyze the way in which the energy of vertex behaves when successive coalescences are made at different vertices of the graph. Specifically, it shows that if we do not coalesce in an "organized" way by controlling the degrees and where the coalescence will be made, we can obtain different results.

At this point, we already have examples where the behavior is maintained; thus, curiosity is aroused to explore other possibilities. Towards this aim, if we apply the Theorem \ref{t3} to the graph $G_{n}$ we 
get:

$$\mathcal{E}_{G_{n}}(v) \leq 1 + \sum\limits_{i=1}^{n} \frac{1}{\sqrt{d_{i}+1}}.
$$

Therefore, if the series converges, we have that $\mathcal{E}_{G_{n}}(v) \nrightarrow \infty$ when $n \rightarrow \infty$. Furthermore, the following is corroborated: $\mathcal{E}_{G_{n}}(u) \nrightarrow  0$, when $n \rightarrow \infty$, according to Lemma \ref{p10}. Thus, with the help of the series we have a sufficient condition to find behaviors totally opposite to that of Theorem \ref{t2}. It is worthwhile  to note that the conditions found so far for obtaining  the same or totally opposite behavior are mutually exclusive, as expected. This is because, the series is lower bounded by $n/(d(n)+1)$.

\section{Conclusion}
In this study, we examined the effect on the energy per vertex of a tree after it was coalesced with a bipartite graph. In particular, it was shown that the energies decrease and increase if the vertices are at odd and even distance, respectively. Thus, an alternating behavior similar to that frequently observed in physical and chemical properties of organic compounds was evidenced. On the other hand, when analyzing the long-term impact produced when successive coalescences are made with respect to a fixed vertex, it was concluded that, in some cases, the edges incident on the vertex tend to disappear. In fact, by interpreting the energy of vertices as a payoff vector in a cooperative game, a possible explanation of the aforementioned behaviors can be obtained. Finally, it would be pertinent to examine the behavior of energy per vertex in bipartite graphs, without limiting it only to the merging vertex or to vertices that lie on paths formed by edge cuts.  We hope to develop this in the future and, at the same time, demonstrate the usefulness of studying vertex energy using the game theory approach.

\subsection*{Acknowledgment}

The authors were supported by CONACYT Grant CB-2017-2018-A1-S-9764.

\end{document}